\documentclass[11pt]{article}

\usepackage{amsmath,amssymb,amsthm,mathtools,setspace,graphicx,array,tikz,url}
\usepackage[text={6.5in,8.4in}, top=1.3in]{geometry}
\usepackage[rightComments=false]{algpseudocodex}
\usepackage{algorithm}
\newtheorem{theorem}{Theorem}[section]
\newtheorem{proposition}[theorem]{Proposition}
\newtheorem{lemma}[theorem]{Lemma}
\newtheorem{corollary}[theorem]{Corollary}
\newtheorem{conjecture}[theorem]{Conjecture}
\newtheorem{observation}[theorem]{Observation}
\newtheorem{problem}[theorem]{Problem}
\newtheorem{definition}[theorem]{Definition}

\newcommand{\T}{{\mathcal T}}
\newcommand{\eps}{{\varepsilon}}

\begin{document}
	
\title{
	On tournament inversion
}
	
\author{
	Raphael Yuster
	\thanks{Department of Mathematics, University of Haifa, Haifa 3498838, Israel. Email: raphael.yuster@gmail.com\,.}
}
	
\date{}
	
\maketitle
	
\setcounter{page}{1}
	
\begin{abstract}
	An {\em inversion} of a tournament $T$ is obtained by reversing the direction of all edges with both endpoints in some set of vertices. Let ${\rm inv}_k(T)$ be the minimum length of a sequence of inversions using sets of size at most $k$ that result in the transitive tournament. Let ${\rm inv}_k(n)$ be the maximum of ${\rm inv}_k(T)$ taken over $n$-vertex tournaments.
	It is well-known that ${\rm inv}_2(n)=(1+o(1))n^2/4$ and it was recently proved by Alon et al.\! that ${\rm inv}(n):={\rm inv}_{n}(n)=n(1+o(1))$. In these two extreme cases ($k=2$ and $k=n$), random tournaments are extremal objects. It is proved that ${\rm inv}_k(n)$ is {\em not} attained by random tournaments when
	$k \ge k_0$ and conjectured that ${\rm inv}_3(n)$ is (only) attained by (quasi) random tournaments.
	It is further proved that
	$(1+o(1)){\rm inv}_3(n)/n^2 \in [\frac{1}{12}, 0.0992)$ and
	$(1+o(1)){\rm inv}_k(n)/n^2 \in [\frac{1}{2k(k-1)}+\delta_k, \frac{1}{2 \lfloor k^2/2 \rfloor}-\epsilon_k]$
	where $\epsilon_k > 0$ for all $k \ge 3$ and $\delta_k > 0$ for all $k \ge k_0$.

\end{abstract}

\section{Introduction}

In this paper we mainly consider oriented graphs, which are digraphs without loops, digons, or parallel edges.
In particular, we consider tournaments, which are oriented complete graphs.
For an oriented graph $D$ and a set $X \subseteq V(D)$, the {\em inversion} of $X$ in $D$ is the oriented graph obtained from
$D$ by reversing the direction of the edges with both endpoints in $X$; synonymously, we view an inversion as an {\em operation} on $D$ and say that we {\em invert} $X$ in $D$.
Notice that inverting a sequence $X_1,\ldots,X_t$ results in the same oriented graph for any permutation of that sequence. If inverting a sequence $X_1,\ldots,X_t$ results in an acyclic digraph, we say that
$\{X_1,\ldots,X_t\}$ forms a {\em decycling set} of $D$. The {\em inversion number of $D$}, denoted
${\rm inv}(D)$, is the minimum size of a decycling set. If each element of a decycling set has size at most $k$,
we say that the set is {\em $k$-decycling} and let ${\rm inv}_k(D)$ denote the minimum size of a $k$-decycling set. We observe that ${\rm inv}(D)={\rm inv}_{n}(D)$ where $|V(D)|=n$, and that ${\rm inv}_2(D)$ is the
size of a minimum feedback edge set of $D$.
The extremal parameters of interest are ${\rm inv}(n)={\rm inv}_n(n)$ and ${\rm inv}_k(n)$ which, respectively, are the maximum of ${\rm inv}(D)$ and ${\rm inv}_k(D)$ taken over all oriented graphs with $n$ vertices.
When studying these extremal parameters, we may and will restrict to the class of $n$-vertex tournaments,
as adding edges to $D$ cannot decrease ${\rm inv}_k(D)$.

The parameter ${\rm inv}_2(n)$ is asymptotically well-understood.
It is straightforward that any digraph can be made acyclic by inverting (equivalently, removing) at most half of its edges. Spencer \cite{spencer-1971} proved that ${\rm inv}_2(n) \le \frac{1}{2}\binom{n}{2} - \Omega(n^{3/2})$.
A probabilistic construction of Spencer \cite{spencer-1980}, later simplified
with an improved constant by de la Vega \cite{delavega-1983}, shows that
${\rm inv}_2(n) \ge \frac{1}{2}\binom{n}{2} - O(n^{3/2})$, hence ${\rm inv}_2(n)=(1+o(1))n^2/4$ and the growth rate below the $n^2/4$ threshold is $\Theta(n^{3/2})$.
In fact, it was shown by Chung and Graham \cite{CG-1991} that an $n$-vertex tournament $T$  has ${\rm inv}_2(T)=(1+o(1))n^2/4$ if and only if it is quasi-random \footnote{More formally, when discussing quasi-randomness we need to consider infinite {\em sequences} of tournaments; see \cite{CG-1991}.}.

The study of ${\rm inv}(D)$ was initiated by Belkhechine \cite{belkhechine-2009} and followed by several papers
that considered its computational and extremal aspects \cite{BSH-2022,BBBP-2010,TSP-2022}.
It is not difficult to show that ${\rm inv}(n) \le n(1+o(1))$ \cite{APSSW-2022,BSH-2022,BBBP-2010} but obtaining an asymptotically matching lower bound is more involved and was only recently resolved by Alon et al.\! \cite{APSSW-2022} who proved that
${\rm inv}(n) \ge n - \sqrt{2n\log(n)}$ \footnote{Whenever the base of a logarithm is not specified, it is assumed to be $2$.} for large $n$, thus establishing that ${\rm inv}(n)=n(1+o_n(1))$.
Moreover, their proof shows that a random $n$-vertex tournament $T$ almost surely has
${\rm inv}(T)=n(1+o(1))$.

As the extreme cases ($k=2$ and $k=n$) are solved up to low order terms, and as random tournaments are extremal objects in both of these cases, one wonders what happens for other $k$.
Mainly, what is the asymptotic behavior of ${\rm inv}_k(n)$? Is it always the case that random tournaments are extremal objects?

Fix $k \ge 3$. It is readily seen that ${\rm inv}_k(n) \ge (1+o(1))n^2/2k(k-1)$.
Indeed, as the minimum feedback edge set of a tournament $T$ has size ${\rm inv_2}(T)$, and since each element in a $k$-decycling set of $T$ changes the direction of at most $\binom{k}{2}$ edges, we have that
${\rm inv}_k(T) \ge {\rm inv_2}(T)/\binom{k}{2}$ and the claim holds by recalling that
${\rm inv}_2(n)=(1+o_n(1))n^2/4$. In fact, it is not difficult to prove that random tournaments attain this bound whp \footnote{Throughout this paper, {\em whp} means ``with probability tending to one as $n$ tends to infinity''.}, as the following proposition shows.
\begin{proposition}\label{prop:random}
	For a random $n$-vertex tournament $T$, ${\rm inv}_k(T) = (1+o(1))n^2/2k(k-1)$ whp.
\end{proposition}
\noindent
We conjecture that this bound is asymptotically tight for $k=3$ and is attained only by quasi-random
(hence also by random) tournaments.
\begin{conjecture}\label{conj:1}
	${\rm inv}_3(n) = (1+o(1))\frac{n^2}{12}$. Moreover, an $n$-vertex tournament $T$ has 
	${\rm inv}_3(T) = (1+o(1))\frac{n^2}{12}$ if and only if it is quasi-random.
\end{conjecture}
\noindent
Note that Conjecture \ref{conj:1} actually consists of two distinct assertions and a proof of each one does not necessarily imply the other. Determining the asymptotic behavior of ${\rm inv}_3(n)$ seems challenging.
Our first main result gives an upper bound which is not far from the conjectured value.
\begin{theorem}\label{t:triangle}
	${\rm inv}_3(n) \le \frac{257}{2592}(1+o(1))n^2$\;.
\end{theorem}
\noindent
Since ${\rm inv}_3(n) \ge (1+o(1))\frac{n^2}{12}$, we have that $(1+o(1)){\rm inv}_3(n)/n^2 \in [\frac{1}{12}, 0.0992)$.

\vspace{5pt}
As for larger fixed $k$, we are able to show that starting from some given $k_0$, the lower bound
$(1+o(1))n^2/2k(k-1)$ is, perhaps surprisingly, {\em not} tight.
This is a consequence of the proof of our second main result, which gives upper and lower bounds 
for ${\rm inv}_k(n)$.
\begin{theorem}\label{t:inv-k}
	$(1+o(1)){\rm inv}_k(n)/n^2 \in [\frac{1}{2k(k-1)}+\delta_k, \frac{1}{2\lfloor k^2/2 \rfloor}-\epsilon_k]$
	where $\epsilon_k > 0$ for all $k \ge 3$ and $\delta_k > 0$ for all $k \ge k_0$.
\end{theorem}
Note that whenever $\delta_k > 0$, Proposition \ref{prop:random} implies that random tournament
are {\em not} the extremal objects. Thus, by Theorem \ref{t:inv-k}, for $k \ge k_0$, tournaments that attain ${\rm inv}_k(n)$ are far from random. On the other hand, notice that Conjecture \ref{conj:1} asserts that
$\delta_3 = 0$ (and recall that $\delta_2 = 0$).
\begin{problem}\label{prob:2}
		Find the smallest $k_0$ for which $\delta_{k_0} > 0$.
\end{problem}

The rest of this paper proceeds as follows: In Section \ref{sec:prelim} we introduce some definitions and collect some known tools that are needed for the proofs. We consider the case $k=3$ in Section \ref{t:triangle} where we prove Theorem \ref{t:triangle}. In Section \ref{sec:cliques} we consider larger $k$ and prove
Theorem \ref{t:inv-k}.

\section{Preliminaries}\label{sec:prelim}

This section presents several definitions and tools required for the proofs of our main results.

\subsection{Hypergraph coloring}\label{subsec:hyper}

Recall that a $k$-uniform hypergraph is a collection of $k$-sets (the edges) of some $n$-set
(the vertices). The {\em degree} $d(x)$ of a vertex $x$ in a hypergraph is the number of
edges containing $x$ and the {\em co-degree} of a pair of distinct vertices $x,y$ is the number of edges containing both.
A {\em matching} in a hypergraph is a set of pairwise 
disjoint edges. The {\em chromatic index} of a hypergraph $H$,
denoted $\chi'(H)$, is the smallest integer $q$ such that the set of edges of $H$ can be partitioned
into $q$ matchings.
The following result of Pippenger and Spencer \cite{PS-1989}
gives sufficient conditions on $H$ which guarantee that $\chi'(H)$ is close to the maximum degree of $H$.
\begin{lemma}[\cite{PS-1989}]\label{l:ps}
	For an integer $k \geq 3$ and a real $\gamma > 0$ 
	there exists a real $\beta=\beta(k,\gamma)$ so that the following holds:
	If a $k$-uniform hypergraph $H$ has the following properties for some $t$:\\
	(i) $(1-\beta)t < d(x) < (1+\beta)t$ holds for all vertices,\\
	(ii) $d(x,y) < \beta t$ for all distinct $x$ and $y$,\\
	then $\chi'(H) \leq (1+\gamma)t$. \qed
\end{lemma}

\subsection{Digraphs, permutations, and random graphs}\label{subsec:tourperm}

The edge-set of every digraph $D$ is the disjoint union of the edge sets of two directed acyclic subgraphs.
Indeed, consider some permutation $\pi$ of $V(D)$.
Let $D_L(\pi)$ be the spanning subgraph of $D$ where $(i,j) \in E(D_L(\pi))$ if and only if $(i,j) \in E(D)$ and $\pi(i) < \pi(j)$.
Let $D_R(\pi)$ be the spanning subgraph of $D$ where $(i,j)\in E(D_R(\pi))$ if and only if $(i,j) \in E(D)$ and $\pi(i) > \pi(j)$.
Since $E(D_R(\pi)) \cup E(D_L(\pi)) = E(D)$, we can cover the edges of $D$ using just two
directed acyclic subgraphs of $D$.
When referring to $D_L(\pi)$ and $D_R(\pi)$ it it convenient to view them as {\em undirected} simple graphs,
but recalling that they correspond to edges of $T$ going from left to right in the case of $D_L(\pi)$ or
from right to left in the case of $D_R(\pi)$.

An $n$-vertex  {\em random tournament} is the probability space $T(n)$ of tournaments on vertex set $[n]$, obtained by orienting the edges of $K_n$ at random (i.e., each direction is decided with a fair coin flip) and all $\binom{n}{2}$ choices are independent. By definition, for each given permutation $\pi$ of $[n]$,
if $T \sim T(n)$ then each of $T_L(\pi)$ and $T_R(\pi)$ is distributed as the binomial random graph $G(n,\frac{1}{2})$.
\begin{proof}[Proof of Proposition \ref{prop:random}]
	Fix a permutation $\pi$ of $[n]$. Let $T \sim T(n)$ and notice that since  $T_L(\pi) \sim G(n,\frac{1}{2})$,
	it has $(1+o(1))n^2/4$ edges, whp. By the result of Frankl and R\"odl \cite{FR-1985},
	$G(n,\frac{1}{2})$ can be almost entirely packed with pairwise edge-disjoint copies of $K_k$, whp.
	Equivalently, this means that whp we can find a collection $C$ of $((1+o(1))n^2/4)/\binom{k}{2}$ sets of vertices, each of size either $k$ or $2$, such that (i) any pair of sets in $C$ intersect in at most one vertex; (ii) each edge of $T_L(\pi)$ is contained in precisely one set of $C$; (iii) each edge of $T_R(\pi)$ is not contained in any set of $C$. Therefore, $C$ forms a decycling set of $T$.
\end{proof}

\subsection{Fractional packing}\label{subsec:fractional}

For an undirected graph $G$, let $\binom {G}{k}$ denote the set of all $K_k$ copies of $G$ (namely, subgraphs of $G$ that are isomorphic to $K_k$).
A function $\phi$ from $\binom {G}{k}$ to $[0,1]$ is a {\em fractional $K_k$-packing} of $G$ if for each edge of $G$, the sum of the values of $\phi$ taken over all $K_k$-copies that contain that edge is at most $1$.
The {\em value} of $\phi$ is
$$
|\phi|=\sum_{H \in \binom{G}{k}} \phi(H)
$$
and $\nu^*_k(G)$ is the maximum of $|\phi|$ taken over all fractional $K_k$-packings of $G$.
A {\em $K_k$-packing} of $G$ is a fractional $K_k$-packing whose image is $\{0,1\}$.
Equivalently, it is a set of pairwise edge-disjoint copies of $K_k$.
Letting $\nu_k(G)$ denote the maximum value of a $K_k$-packing of $G$,
we have $\nu^*_k(G) \geq \nu_k(G)$.
An important result of Haxell and R\"odl \cite{HR-2001} (see also \cite{yuster-2005}) shows that the converse inequality holds up to an additive error term which is negligible for dense graphs.
\begin{lemma}[\cite{HR-2001}]\label{l:hr}
	For every $\eps > 0$ and for every positive integer $k \ge 3$ there exists $N=N(k,\eps)$ such that
	for any graph $G$ with $n > N$ vertices, $\nu^*_k(G)-\nu_k(G) \le \eps n^2$. \qed
\end{lemma}

\section{Triangle inversions}\label{sec:triangle}

In this section we prove Theorem \ref{t:triangle}. We begin with a lemma that will be useful to ``finish off''
decycling a tournament which already has a relatively small feedback edge set, using triangle inversions.

\begin{lemma}\label{l:c4}
	Let $T^*$ be a tournament and suppose that $T^*_L(\pi)$ has a (not necessarily induced) $4$-cycle on vertices $a,b,c,d$ where $\pi(a) < \pi(b) < \pi(c) < \pi(d)$. Then, there are two sets of vertices $X,Y$ of three vertices each, such that inverting $\{X,Y\}$ reverses the direction of the edges of the $4$-cycle without	affecting the direction of any other edge of $T^*$.
\end{lemma}
\begin{proof}
	Since $T^*_L(\pi)$ corresponds to an acyclic subgraph of $T^*$, there are exactly three possible
	orientations of the edges of the four-cycle in $T^*$. These are: (i) $\{(a,b), (a,c), (b,d), (c,d)\}$;
	(ii) $\{(a,b),(b,c),(c,d),(a,d)\}$; (iii) $\{(a,c),(b,c),(a,d),(b,d)\}$.
	In the first case, let $X=\{a,b,c\}$ and $Y=\{b,c,d\}$.
	In  the second case, let $X=\{a,b,c\}$ and $Y=\{a,c,d\}$.
	In  the third case, let $X=\{a,b,c\}$ and $Y=\{a,b,d\}$.
	In each of these cases we see that inverting $\{X,Y\}$ reverses the direction of the edges of the
	corresponding $4$-cycle without affecting the direction of any other edge of $T^*$.
\end{proof}

\begin{corollary}\label{coro:triangles}
	Let $T^*$ be an $n$-vertex tournament and suppose that $|E(T^*_L(\pi))| \le \alpha n^2$.
	Then, ${\rm inv}_3(T^*) \le \alpha n^2/2+o(n^2)$. 
\end{corollary}
\begin{proof}
	By Lemma \ref{l:c4} we can repeatedly invert pairs of three-sets of vertices each, 
	until we obtain a tournament $T^{**}$ for which $T^{**}_L(\pi)$ has no four-cycle.
	The number of inversions performed starting at $T^*$ and arriving at $T^{**}$ is
	therefore precisely $(|E(T^*_L(\pi))|-|E(T^{**}_L(\pi))|)/2$. By the Kov\'ari-S\'os-Tur\'an Theorem \cite{KST-1954}, we have
	$|E(T^{**}_L(\pi))| \le (n^{3/2}+n)/2$, hence the corollary.
\end{proof}

A high level approach to proving Theorem \ref{t:triangle} follows. Suppose that $T$ is a given $n$-vertex tournament. We prove that a randomly chosen permutation $\pi$ has the following two properties whp: (i) $T_L(\pi)$ has roughly $n^2/4$ edges; (ii) $T_L(\pi)$ has many (say, roughly $\beta n^2$) pairwise edge-disjoint triangles. Once we show these properties, we can reverse the claimed set of pairwise edge-disjoint triangles to obtain a tournament
$T^*$ for which $T^*_L(\pi)$ has roughly $\alpha n^2$ edges where $\alpha=\frac{1}{4}-3\beta$.
We then apply Corollary \ref{coro:triangles} to reverse the edges of $T^*_L(\pi)$ and obtain a decycling
set of $T$ of order roughly $\beta n^2 + \alpha n^2/2 = (\frac{1}{8}-\frac{\beta}{2})n^2$.
Proving that a random permutation satisfies (i) is a standard argument, but proving (ii) for a reasonable value of $\beta$ requires careful analysis, detailed later. The main difficulty stems from the fact
that we do not know much about the structure of $T_L(\pi)$.

For the rest of this section we fix a tournament $T$ on vertex set $[n]$. We start by proving (i) above.
\begin{lemma}\label{l:randpi-i}
	Let $\pi$ be a randomly chosen permutation of $[n]$. Whp,
	$|E(T_L(\pi))|=(1+o(1))n^2/4$.
\end{lemma}
\begin{proof}
	Let $X(i,j)$ denote the indicator random variable for the event that $ij \in E(T_L(\pi))$.
	We observe that $X(i,j) \sim {\rm  Bernoulli}(\frac{1}{2})$ and that
	$|E(T_L(\pi))|$ is the sum of $X(i,j)$ taken over all edges $(i,j) \in E(T)$, so
	its expected value is $\binom{n}{2}/2=(1+o(1))n^2/4$. To show that $|E(T_L(\pi))|$ is concentrated, we upper-bound its variance. Notice that if $\{i,j\} \cap \{i',j'\} = \emptyset$, then $X(i,j)$ and $X(i',j')$ are independent
	and notice that there are fewer than $n^3$ ordered pairs $(\{i,j\},\{i',j'\})$ for which
	$\{i,j\} \cap \{i',j'\} \neq \emptyset$. Hence, ${\rm Var}[|E(T_L(\pi))|] \le {\mathbb E}[|E(T_L(\pi))|] + n^3= o({\mathbb E}[|E(T_L(\pi))|]^2)$. The lemma now follows from Chebyshev's inequality.
\end{proof}

We now turn to our second task, i.e., showing that $T_L(\pi)$ has many pairwise edge-disjoint triangles, whp.
Let us first see a way to obtain some nontrivial bound for this quantity.
In a recent paper, Gruslys and Letzter \cite{GL-2020}, improving an earlier result of Keevash and Sudakov \cite{KS-2004}, and thereby confirming a conjecture of Erd\H{o}s, proved that in any two-coloring of the
edges of $K_n$, there are at least $(1-o(1))n^2/12$ pairwise edge-disjoint monochromatic triangles.
Observing that coloring the edges of $T_L(\pi)$ red and coloring the edges of $T_R(\pi) = T_L(\pi^{reverese})$ blue corresponds to a two-coloring of the edges of $K_n$, we have by Lemma \ref{l:randpi-i} and by
the result in \cite{GL-2020} that there exists a permutation $\pi$ such that $|E(T_L(\pi))|=(1+o(1))n^2/4$ and $T_L(\pi)$ has at least $(1+o(1))n^2/24$ pairwise edge-disjoint triangles. Notice that in this argument, we are only using a weaker form of the result in \cite{GL-2020}: that in any two-coloring of the
edges of $K_n$, there are many pairwise edge-disjoint monochromatic triangles all of the same color.
In fact, it was conjectured by Jacobson (see \cite{EFGJL-2001}) that there are always at least $(1+o(1))n^2/20$
such triangles (there are examples showing that, if true, the constant $\frac{1}{20}$ is optimal), but this is still open, although it was proved in \cite{GL-2020} that the correct constant
for this question must be strictly larger than $\frac{1}{24}$.
Notice, however, that even applying the latter form of the question (i.e., Jacobson's conjectured value) may be stronger than what we need. Indeed, we will show that this is provably the case;
we will show that whp a random permutation gives that $T_L(\pi)$ has $(1+o(1))\beta n^2$ edge-disjoint triangles,
where the obtained constant $\beta$ is significantly larger than (the proven) $\frac{1}{24}$ and (the conjectured) $\frac{1}{20}$ constants of the aforementioned monochromatic triangles questions.

We require some further notation and definitions. Let $q \ge 4$ be an integer parameter to be set later
(it will be small, but not too small).
Let $\T_q$ be the set of all tournaments on $q$ vertices. For a tournament $Q \in \T_q$ and a permutation $\sigma$ of
$V(Q)$, recall that $\nu^*_3(Q_L(\sigma))$ is the maximum value of a fractional triangle packing of $Q_L(\sigma)$.
\begin{definition}
	Let $P$ be a family of permutations of $V(Q)$. Let ${\rm avg}_P(Q)$ be the average of
	$\nu^*_3(Q_L(\sigma))$ where $\sigma$ is taken over all permutations in $P$.
	Let ${\rm avg}(Q)$ be the average taken over all possible $q!$ permutations.
\end{definition}
It is possible that for some families, ${\rm avg}_P(Q)$ is larger than the overall average
${\rm avg}(Q)$ while for other families, it is smaller. It will be beneficial to assign to each tournament $Q \in \T_q$, a family $P$ for which ${\rm avg}_P(Q)$ is as large as possible, under some restrictions whose
necessity will be apparent later (for instance, to make computations feasible, we would like $|P|$ to be small).
To formalize these restrictions, we need the following definition.
\begin{definition}
	Let $P$ be a family of permutations of a set $X$ with $|X|=q$.
	We call $P$ {\em orthogonal} if for any ordered pair $u,v$ of elements of $X$ and any two positions
	$1 \le i < j \le q$, there is exactly one $\sigma \in P$ such that $\sigma(u)=i$ and $\sigma(v)=j$.
\end{definition}
We note that the definition implies that the size of an orthogonal family is $q(q-1)$ (namely, much smaller than $q!$).
For a prime power $q$, a construction of an orthogonal family of permutations can be obtained from certain constructions of $q-1$ mutually orthogonal Latin squares (aka MOLS). For example, for $q=9$, such a family is shown in Table \ref{table:orthogonal} and is obtained from the (columns of the) $8$ pairwise orthogonal Latin squares of order $9$ given in \cite{CD-2007}, Page 164.
An obvious but useful observation is that if $P$ is an orthogonal family of permutations of $X$, and 
$\sigma^*$ is any permutation of $[q]$, the family $\{\sigma^*\sigma\,:\, \sigma \in P\}$ of permutations of $X$ is also orthogonal. In particular, this means that if there is an orthogonal family of permutations of $V(Q)$ for a tournament $Q$, then there is some such orthogonal family $P$ for which ${\rm avg}_P(Q) \ge {\rm avg}(Q)$.

\begin{table}[ht]
	\centering
	\begin{tabular}{cccccc}
		$012345678$
		& $120453786$
		& $201534867$
		& $345678012$
		& $453786120$
		& $534867201$
		
		\\$678012345$
		& $786120453$
		& $867201534$
		& $021687354$
		& $102768435$
		& $210876543$
		
		\\$354021687$
		& $435102768$
		& $543210876$
		& $687354021$
		& $768435102$
		& $876543210$
		
		\\$036471825$
		& $147582603$
		& $258360714$
		& $360714258$
		& $471825036$
		& $582603147$
		
		\\$603147582$
		& $714258360$
		& $825036471$
		& $048723561$
		& $156804372$
		& $237615480$
		
		\\$372156804$
		& $480237615$
		& $561048723$
		& $615480237$
		& $723561048$
		& $804372156$
		
		\\$057138246$
		& $138246057$
		& $246057138$
		& $381462570$
		& $462570381$
		& $570381462$
		
		\\$624705813$
		& $705813624$
		& $813624705$
		& $063852417$
		& $174630528$
		& $285741306$
		
		\\$306285741$
		& $417063852$
		& $528174630$
		& $630528174$
		& $741306285$
		& $852417063$
		
		\\$075264183$
		& $183075264$
		& $264183075$
		& $318507426$
		& $426318507$
		& $507426318$
		
		\\$642831750$
		& $750642831$
		& $831750642$
		& $084516732$
		& $165327840$
		& $273408651$
		
		\\$327840165$
		& $408651273$
		& $516732084$
		& $651273408$
		& $732084516$
		& $840165327$
	\end{tabular}
	\caption{An orthogonal family of permutations of $X=\{0,\ldots,8\}$.}
	\label{table:orthogonal} 
\end{table}  

Hereafter we assume that $q$ is such that a family of orthogonal permutations of sets of size $q$ exists (e.g., if $q=9$ this holds by Table \ref{table:orthogonal}). Let $\zeta \ge 0$ be a real value to be computed later.
Suppose that for each $Q \in \T_q$ we can find an orthogonal family of permutations $P=P(Q)$ (note: distinct $Q$ may be assigned distinct $P$) such that ${\rm avg}_{P(Q)}(Q) \ge \zeta$. We will show how to lower-bound the number of pairwise edge-disjoint triangles in $T_L(\pi)$ in terms of $\zeta$ and $n$.
To this end, we need to define a certain hypergraph, which depends on $T$, on $\pi$, and on the assignments
$P(Q)$ for each $Q \in \T_q$. We next formalize the definition of this hypergraph (Definition \ref{def:hyper}).
\begin{definition}
For each $q$-subset $X$ of vertices of $T$, consider the sub-tournament $T[X] \in \T_q$ it induces.
We say that $X$ is successful if $\sigma = \pi|_X \in P(T[X])$. Otherwise, it is unsuccessful.
\end{definition}
\begin{observation}\label{obs:1}
Since $\pi$ is a random permutation, so is its restriction $\sigma=\pi|_X$, hence the probability that $X$ is successful is precisely $P(T[X])/q! = 1/(q-2)!$.
\end{observation}
\begin{definition}\label{def:hyper}
	Let $H(T,\pi)$ be the hypergraph whose vertex set is $E(T)$ and for each successful
	$q$-subset $X$ of vertices of $T$, there is an edge of $H$ consisting of all
	the edges of $T[X]$.
\end{definition}
\noindent
Notice that $H(T,\pi)$ is $\binom{q}{2}$-uniform. The following lemma establishes some important properties of $H(T,\pi)$ that hold whp. Its proof relies crucially on the orthogonality property of $P(T[X])$.
\begin{lemma}\label{l:hyp-properties}
	Let $\pi$ be a randomly chosen permutation of $[n]$. Whp,
	$H(T,\pi)$ has an induced subhypergraph $H'$ with at least $\binom{n}{2} - 3n^{1.9}$ vertices such that:\\
	(i) The degree of each vertex of $H'$ is $(1+o(1))\frac{\binom{n-2}{q-2}}{(q-2)!}$\,;\\
	(ii) The co-degree of each pair of vertices of $H'$ is at most $n^{q-3}$.
\end{lemma}
\begin{proof}
	We start with the second assertion, which is not probabilistic. Consider two vertices of $H(T,\pi)$,
	i.e., two edges of $T$, say $(u,v)$ and $(w,z)$. The total number of $q$-subsets of vertices of $T$
	that contain both of these edges is $\binom{n-4}{q-4}$ if $\{u,v\} \cap \{w,z\} = \emptyset$ and is
	$\binom{n-3}{q-3}$ if $\{u,v\} \cap \{w,z\} \neq \emptyset$. In any case, we see that the number of $q$-sets
	of vertices containing both $(u,v)$ and $(w,z)$ is less than  $n^{q-3}$ and in particular, the co-degree of $(u,v)$ and $(w,z)$ in $H(T,\pi)$, which only counts successful $q$-sets, is less than $n^{q-3}$.
	
	For the first assertion, fix a vertex of $H(T,\pi)$, i.e., an edge $e=(u,v)$ of $T$.
	Let $d(e)$ be the random variable corresponding to the degree of $e$ in $H(T,\pi)$.
	Let ${\mathcal X}$ be the $q$-sets of vertices of $T$ that contain both $u$ and $v$
	and observe that $|{\mathcal X}|=\binom{n-2}{q-2}$. For $X \in {\mathcal X}$, consider the
	indicator random variable $d(X)$ for the event that $X$ is successful. We have that $d(e) = \sum_{X \in {\mathcal X}} d(X)$. By Observation \ref{obs:1}, we have that
	$d(X)  \sim {\rm  Bernoulli}(\frac{1}{(q-2)!})$ so we obtain that
	$$
	{\mathbb E}[d(e)] = \frac{\binom{n-2}{q-2}}{(q-2)!}\;.
	$$
	We show that $d(e)$ is concentrated by considering its variance. To this end, fix two elements
	of ${\mathcal X}$, say $X$ and $Y$, and consider ${\rm Cov}(d(X),d(Y))$. Notice that as each of $X$ and $Y$ contain both $u$ and $v$, we have that $|X \cap Y| \ge 2$. Now, if $|X \cap Y| \ge 3$ we shall use the trivial bound ${\rm Cov}(d(X),d(Y)) \le 1$ (recall that $d(X)$ and $d(Y)$ are indicators). So, suppose that $|X \cap Y| = 2$, i.e., they to not have common elements other
	than $u$ and $v$. Now, suppose that we are given $d(X)$, i.e., we are told whether $X$ is successful or not.
	Moreover, suppose that we are revealed all the values of $\pi$ on $V(T) \setminus (Y \setminus \{u,v\})$ 	(notice that this informations reveals all the values of $\pi$ on $X$, so in particular, it reveals $d(X)$). So, we know $\pi(u)$ and $\pi(v)$, we know the positions in $\pi$ occupied by $Y \setminus \{u,v\}$ but we do not know the internal ordering of the elements of $Y \setminus \{u,v\}$ within these positions.
	As the family $P(T[Y])$ is orthogonal, there is precisely {\em one} possible ordering of $Y \setminus \{u,v\}$
	for which $\pi|_Y \in P(T[Y])$, i.e., for which $Y$ is successful. Thus, the probability that
	$Y$ is successful, given $d(X)$, is precisely $1/(q-2)!$, i.e., the same a priori probability.
	In particular, ${\rm Cov}(d(X),d(Y))=0$.
	We now have
	\begin{align*}
	{\rm Var}[d(e)] & \le {\mathbb E}[d(e)]+ 2\sum_{t=2}^{q-1} \sum_{\substack{X,Y \in {\mathcal X}\\ |X \cap Y|=t}} {\rm Cov}(d(X),d(Y))\\
	& = {\mathbb E}[d(e)]+ 2\sum_{t=3}^{q-1} \sum_{\substack{X,Y \in {\mathcal X}\\ |X \cap Y|=t}} {\rm Cov}(d(X),d(Y))\\
	& \le {\mathbb E}[d(e)]+ 2\sum_{t=3}^{q-1} n^{2q-2-t}\\
	& \le {\mathbb E}[d(e)]+ 2qn^{2q-5}\\
	& \le 3qn^{2q-5}
	\end{align*}
where in the last inequality we have used that $q \ge 4$. We may now apply Chebyshev's inequality and obtain that
	$$
	\Pr\left[\,|d(e) - {\mathbb E}[d(e)] \,| \ge n^{q-2.1} \right] \le \frac{3qn^{2q-5}}{n^{2q-4.2}}= \Theta(n^{-0.8})\;.
	$$
	As $T$ has fewer than $n^2$ edges, we obtain from the last inequality and from Markov's inequality that whp, all but $O(n^{1.2}) < n^{1.5}$ vertices $e$ of $H(T,\pi)$ have
	$|d(e) - {\mathbb E}[d(e)]| \le n^{q-2.1}$.
	Consider then the set $F$ of at most $n^{1.5}$ vertices $e$ of $H(T,\pi)$ which violate the last inequality.
	For a vertex $v \in V(T)$ we say that $v$ is {\em dangerous} if $v$ is an endpoint of at least $n^{0.6}$
	elements of $F$. So, there are fewer than $2n^{0.9}$ dangerous vertices. Remove from $H(T,\pi)$ all elements
	of $F$ and also remove all vertices $e$ of $H(T,\pi)$ such that $e$ is an endpoint of a dangerous vertex of $T$.
	Let $H'$ be the induced subhypergraph of $H(T,\pi)$ obtained after the removal.
	As we remove at most $|F|+2n^{1.9}$ vertices from $H(T,\pi)$, we have that $H'$ contains at least
	$\binom{n}{2}-3n^{1.9}$ vertices. Clearly, the co-degree of any two vertices in $H'$ is
	not larger than it is in $H$. By how much might a degree of a vertex $e=(u,v)$ in $H'$ decrease?
	It might belong to a successful $q$-subset which contains a dangerous vertex of $T$, but there are only at most
	$2n^{0.9}n^{q-3} < 2n^{q-2.1}$ such $q$-subsets. As $u$ and $v$ are non-dangerous, there may additionally be at most $2n^{0.6}$ vertices $(x,y)$ of $H(T,\pi)$ where $(x,y) \in F$ and $|\{x,y\} \cap \{u,v\}|=1$.
	But then these may cause a further reduction of at most $2n^{0.6}n^{q-3} < n^{q-2.1}$ in  the degree of $e$.
	Additionally, it may be that $(u,v)$ belongs to at $q$-subsets which contain an element $(x,y)$ of $F$ such that $\{x,y\} \cap \{u,v\}=\emptyset$. But then these may cause a further reduction of at most $n^{1.5}n^{q-4} < n^{q-2.1}$ in the degree of $e$. It now follows that all the vertices of $H'$ 
	have degree 
	${\mathbb E}[d(e)] \pm 5n^{q-2.1}$, which is ${\mathbb E}[d(e)](1+o(1))$. 
\end{proof}
We use Lemma \ref{l:hyp-properties} to lower-bound the number of pairwise edge-disjoint triangles in $T_L(\pi)$ in terms of $\zeta$ and $n$.
\begin{lemma}\label{l:zeta}
	Suppose that $\zeta \ge 0$ is such that for each $Q \in \T_q$ it holds that ${\rm avg}_{P(Q)}(Q) \ge \zeta$.
	Then whp it holds that $T_L(\pi)$ has at least $(1-o(1))\frac{\zeta n^2}{q(q-1)}$ pairwise edge-disjoint triangles.
\end{lemma}
\begin{proof}
	Let $\pi$ be a randomly chosen permutation of $[n]$. Define the following random variable for each $q$-subset $X$ of vertices of $T$:
	$$
	X^* = \begin{cases}
	 {\rm avg}_{P(T[X])}(T[X]) &  \text{if } X \text{ is successful},\\
	0 & \text{otherwise}\,.
\end{cases}
	$$
	Let $R$ denote the sum of $X^*$ taken over all $q$-subsets of vertices $T$.
	Notice that by its definition, ${\rm avg}_{P(T[X])}(T[X])$ is the expected value of
	$\nu^*_3(T[X]_L(\pi|_X))$ taken over all permutations of the orthogonal family $P(T[X])$, so $R$ is the expected sum of $\nu^*_3(T[X]_L(\pi|_X))$ taken over all successful $q$-subsets of vertices. By Observation \ref{obs:1} and by our assumption that ${\rm avg}_{P(T[X])}(T[X]) \ge \zeta$ we have that
	$$
	{\mathbb E}[R] = \sum_{X} {\mathbb E}[X^*] = \sum_{X} \frac{{\rm avg}_{P(T[X])}(T[X])}{(q-2)!} \ge \binom{n}{q}\frac{\zeta}{(q-2)!}\;.
	$$
	We next show that $R$ is concentrated by considering its variance. To this end, notice that $R$ is a sum of
	$\binom{n}{q}=\Theta(n^q)$ nonnegative random variables, each bounded from above by the constant $q(q-1)/6$. Indeed, a fractional triangle packing of a graph on $q$ vertices cannot be more than $\frac{1}{3}$ of the number of its edges, so ${\rm avg}_{P(T[X])}(T[X]) \le q(q-1)/6$.  To show that ${\rm Var}[R] =o({\mathbb E}[R]^2)$ it is therefore enough to show that the number of pairs $X,Y$ such that ${\rm Cov}(X^*,Y^*) \neq 0$
	is $o(n^{2q})$. Indeed, observe that the permutations $\pi|_X$ and $\pi|_Y$ are independent whenever $X$
	and $Y$ are disjoint sets of vertices. As each $X$ is not disjoint with at most $q\binom{n-1}{q-1}$ possible $Y$, we have that the number of pairs $X,Y$ such that ${\rm Cov}(X^*,Y^*) \neq 0$ is only
	$O(n^{2q-1})$, as required. As we have shown that ${\rm Var}[R] =o({\mathbb E}[R]^2)$, it follows from Chebyshev's inequality that $R-{\mathbb E}[R]$ is $o({\mathbb E}[R])$ whp and in particular,
	$R \ge (1-o(1))\binom{n}{q}\frac{\zeta}{(q-2)!}$ whp.
	
	By Lemma \ref{l:hyp-properties} and from the previous paragraph we have that whp, $\pi$ is such that: (a) $R \ge (1-o(1))\binom{n}{q}\frac{\zeta}{(q-2)!}$; (b)
	$H(T,\pi)$ has an induced subhypergraph $H'$ on at least $\binom{n}{2}-3n^{1.9}$ vertices satisfying both items of Lemma \ref{l:hyp-properties}.
	For the remainder of the proof we assume that $\pi$ is such that (a) and (b) hold.
	
	Notice that $H'$ satisfies the conditions of Lemma \ref{l:ps} with $k=\binom{q}{2}$, $t=\binom{n-2}{q-2}/(q-2)!$ and $\gamma = o(1)$. Indeed, $H'$ has $\binom{n}{2}(1-o(1))$ vertices
	and by Lemma \ref{l:hyp-properties}, the degree of every vertex of $H'$ is $(1+o(1))\binom{n-2}{q-2}/(q-2)!$
	while the co-degree of every pair of vertices of $H'$ is much smaller, only at most $n^{q-3}$.
	So by Lemma \ref{l:ps},  $\chi'(H') \le (1+o(1))\binom{n-2}{q-2}/(q-2)!$. By the definition of $H(T,\pi)$
	and $H'$, this means that there is a subset $W \subseteq E(T)$ of at least $\binom{n}{2}-3n^{1.9}$ edges, such that
	the family of all successful $q$-sets whose edges are entirely contained in $W$ can be partitioned into $\chi'(H')$
	parts, say $M_1,\ldots,M_{\chi'(H')}$ where for each $M_i$, all $q$-subsets contained in it are pairwise edge-disjoint.
	
	Let $R_i$ denote the contribution to $R$ of the $q$-subsets in $M_i$ and let $R_0$ denote the contribution to $R$ of the $q$-subsets that contain an edge in $E(T) \setminus W$. Hence,
	$R=\sum_{i=0}^{\chi'(H')}R_i$. We first observe that $R_0$ is negligible. As there are only at most $3n^{1.9}$ edges not in $W$, there are at most $3n^{1.9}n^{q-2}$ \, $q$-subsets that contribute to $R_0$, and recall that they each contribute at most $q(q-1)/6$, so in total $R_0=o(n^{q})$. Since
	$R \ge (1-o(1))\binom{n}{q}\frac{\zeta}{(q-2)!}$ we may assume wlog that
	$$
	R_1 \ge \frac{1}{\chi'(H')}(1-o(1))\binom{n}{q}\frac{\zeta}{(q-2)!} \ge (1-o(1))\binom{n}{q}\frac{\zeta}{\binom{n-2}{q-2}} = (1-o(1))\frac{\zeta n^2}{q(q-1)}\;.
	$$
	Recalling that the $q$-sets in $M_1$ are pairwise edge disjoint, we have that $T_L(\pi)$ has fractional triangle packing of size $\nu_3^*(T_L(\pi)) \ge (1-o(1))\frac{\zeta n^2}{q(q-1)}$.
	By Lemma \ref{l:hr} we therefore obtain that $\nu_3(T_L(\pi)) \ge (1-o(1))\frac{\zeta n^2}{q(q-1)}$ as well.	
\end{proof}
\begin{corollary}\label{coro:final}
	Let $q \ge 4$ and suppose that for each $Q \in \T_q$ we can find an orthogonal family of permutations $P=P(Q)$ such that ${\rm avg}_{P(Q)}(Q) \ge \zeta$. Then,
	${\rm inv}_3(n) \le (1+o(1)n^2\left(\frac{1}{8}-\frac{\zeta}{2q(q-1)}\right)$.
\end{corollary}
\begin{proof}
	By Lemma \ref{l:randpi-i} and Lemma \ref{l:zeta}, for every  $n$-vertex tournament $T$, a random permutation $\pi$ of $V(T)$ satisfies whp: (i) 	$|E(T_L(\pi))|=(1+o(1))n^2/4$.
	(ii) $\nu_3(T_L(\pi)) \ge (1-o(1))\frac{\zeta n^2}{q(q-1)}$.
	Fixing such a $\pi$ which satisfies both requirements, defining $\beta := \zeta/q(q-1)$
	and using that ${\rm inv}_3(T) \le (1+o(1))n^2(\frac{1}{8}-\frac{\beta}{2})$ as shown in the paragraph following Corollary \ref{coro:triangles}, the present corollary follows.
\end{proof}

Following Corollary \ref{coro:final}, our remaining task is to find $q \ge 4$ and $\zeta$ such that $\zeta/q(q-1)$ is as large as feasibly computable (note: not all $q$ are possible; if $q$ is not a prime power, we do not have a construction of an orthogonal family of permutations). Let us be more formal about our computational task:
Suppose that $q \ge 4$ is such that an orthogonal family of permutations of sets of order $q$ exists.
For $Q \in \T_q$, let $\zeta(Q)$ be the maximum of ${\rm avg}_{P(Q)}(Q)$ where $P$ ranges over all
orthogonal families of permutations of $V(Q)$. Let $\zeta_q$ be the minimum of $\zeta(Q)$ taken over
all $Q \in \T_q$. Hence, we would like to compute $\zeta_q$ or at least obtain a close lower bound for it,
as any such lower bound $\zeta$ can be applied in Corollary \ref{coro:final}.

\begin{algorithm}
	\caption{Computing a lower bound for $\zeta_q$}\label{alg:cap}
	\begin{algorithmic}
		\Require tournament database ${\mathcal D}_q$; orthogonal family $P \in {\mathcal P}_q$
		\State $\zeta \gets \infty$
		\ForAll{$Q \in  {\mathcal D}_q$}
		\State $best \gets -1$
		\State $loop \gets 0$
		\While{$loop < 1000$ {\bf and}  $best < \zeta$}
		\State $loop \gets loop+1$
		\State $\pi \gets$  random permutation of $[q]$
		\State $P^* \gets \{\pi\sigma \,:\, \sigma \in P\}$
		\State $sum \gets 0$
		\ForAll{$\sigma  \in  P^*$}
		\State $G \gets Q_L(\sigma)$ \Comment{Construct the graph $Q_L(\sigma)$}
		\State $sum \gets sum + LPSolve(G)$ \Comment{Compute $\nu^*(Q_L(\sigma))$ via an LP package}
		\EndFor
		\State $avg \gets sum/(q(q-1))$ \Comment{Compute ${\rm avg}_{P^*}(Q)$}
		\If{$avg > best$}
		\State $best \gets avg$
		\EndIf
		\EndWhile
		\If{$best < \zeta$}
		\State $\zeta \gets best$
		\EndIf
		\EndFor
		\State \Return $\zeta$
	\end{algorithmic}
\end{algorithm}

So, suppose we are given a database ${\mathcal D}_q$ of all tournaments on $q$ vertices, each labeled on $[q]$, i.e., for each $Q \in \T_q$, there is precisely one element of ${\mathcal D}_q$ that is isomorphic to it.
For example, such a database for all $q \le 10$ is given in \url{https://users.cecs.anu.edu.au/~bdm/data/digraphs.html}. Furthermore, suppose that ${\cal P}_q$ is
the set of all orthogonal families of permutations of $[q]$. Then, for each $Q \in {\mathcal D}_q$,
for each $P \in {\cal P}_q$ and for each $\sigma \in P$, we can easily compute the graph $Q_L(\sigma)$
and then construct a linear program to determine $\nu^*(Q_L(\sigma))$, thus determining ${\rm avg}_{P(Q)}(Q)$, consequently determining $\zeta(Q)$, consequently determining $\zeta_q$.
The number of operations of this approach is at least $|{\mathcal D}_q||{\mathcal P}_q|q(q-1)L_q$
where $L_q$ is the time to run a single linear program; the latter is non-negligible as the program may have
size $\Theta(q^3)$ (the number of possible triangles in $Q_L(\sigma)$).
The values of $|{\mathcal D}_q|$ follow the sequence OEIS A000568 \cite{OEIS}, so we have, e.g.,
$|{\mathcal D}_{11}|=903753248$ and $|{\mathcal D}_9|=191536$.
We see that already for $q=11$ we need to make at least $|{\mathcal P}_{11}|99412857280$ calls to a linear program of nontrivial size (as most calls involve linear programs with over $100$ variables)
which is overwhelmingly huge (even if the search space were to be trimmed by employing some heuristics,
e.g., as we may settle for a lower bound for $\zeta_q$, we don't need to examine all of the utterly huge $|{\mathcal P}_{11}|$,
rather just a few of its members, but even this is not feasible already for a single member).
For $q=9$ (recall, for $q=10$ we do not know of an orthogonal family), the number of calls to a linear program
(again, not a very small one) is $|{\mathcal P}_9|13790592$ which becomes feasible if instead of
going over all $|{\mathcal P}_9|$, we only scan a few members of it. Similarly, we can do the same for
$q=4,5,7,8$. Indeed, this is what our program does; its pseudocode is given in Algorithm \ref{alg:cap}
and its code can be obtained from
\url{https://github.com/raphaelyuster/decycling/blob/main/decycling_latin.cpp}.

As can be seen from the pseudocode, as well as the code, we start with some fixed orthogonal family $P$.
For example, if $q=9$ we use the one in Table \ref{table:orthogonal}. For each $Q \in {\mathcal D}_q$
we generate a constant, say $1000$, orthogonal families of the form
$P^* = \left\{\pi\sigma\,:\, \sigma \in P\right\}$ where $\pi$ is a random permutation of $[q]$
and for each such family we compute $avg_{P^*}(Q)$, taking the best (i.e., highest over all $1000$ trials) result that we find as a lower bound for
$\zeta(Q)$, and setting $\zeta$ to be the minimum of these lower bounds, taken over all $Q \in {\mathcal D}_q$.
We summarize the result of the program runs in Table \ref{table:zeta}.
As can be seen, for $q=9$ we obtain a value of $\zeta=\frac{67}{18}$ which implies 
the constant $\frac{257}{2592}$ for $\frac{1}{8}-\frac{\zeta}{2q(q-1)}$, completing the proof of
Theorem \ref{t:triangle}. \qed

\begin{table}
	\renewcommand{\arraystretch}{1.2}
	\centering
	\begin{tabular}{c||c|c}
		$q$ &  $\zeta$ & $\frac{1}{8}-\frac{\zeta}{2q(q-1)}$ \\
		\hline
		$4$ & $\frac{1}{3}$ & $\frac{1}{9} \approx 0.1111$ \\
		\hline
		$5$ & $\frac{7}{10}$ & $\frac{43}{400} = 0.1075$ \\
		\hline
		$7$ & $\frac{27}{14}$ & $\frac{5}{49} \approx 0.1021$ \\
		\hline
		$8$ & $\frac{153}{56}$ & $\frac{631}{6272} \approx 0.1006$ \\
		\hline
		$9$ & $\frac{67}{18}$& $\frac{257}{2592} \approx 0.0992$
	\end{tabular}
	\caption{The values of $\zeta \le \zeta_q$ for various choices of $q$ obtained from the program whose pseudocode is given in Algorithm \ref{alg:cap}, and the resulting constant from Corollary \ref{coro:final}.}
	\label{table:zeta} 
\end{table} 

\section{Larger $k$}\label{sec:cliques}

In this section we prove Theorem \ref{t:inv-k}. Starting with the lower bound, our aim is to construct an $n$-vertex tournament $T$ for which ${\rm inv}_k(T)$ is large. 
\begin{lemma}\label{l:q}
	For all sufficiently large $k$, there exists a tournament $Q$ on $q=k^2$ vertices such that:\\
	(i) ${\rm inv}_2(Q) \ge \frac{q^2}{4}-2q^{3/2}$\,;\\
	(ii) For any set $R$ of $r$ vertices for which $1.04^r  \ge q$, it holds that ${\rm inv}_2(Q[R]) \ge \frac{1}{4}\binom{r}{2}$\;.
\end{lemma}
\begin{proof}
	Let $Q$ be a random tournament on $q$ vertices. It suffices to prove that each of the two items in the lemma's statement holds with probability at least, say, $\frac{2}{3}$, as then they both hold with positive probability, implying $Q$'s existence.
	As for the first item, the result of de la Vega \cite{delavega-1983} states that
	with probability $1-o_q(1)$ it holds that 
	${\rm inv}_2(Q) \ge \frac{1}{2}\binom{q}{2}-1.73q^{3/2}$, implying in particular that for all $k$ sufficiently large, with probability at least $\frac{2}{3}$ it holds that ${\rm inv}_2(Q) \ge \frac{q^2}{4}-2q^{3/2}$.
	
	For the second item, fix a subset $R \subseteq V(Q)$ with $r=|R|$ where $1.04^r \ge q$.
	As $Q$ is a random tournament, so is its sub-tournament $Q[R]$.
	Let $\pi$ be a permutation of $R$ and consider ${Q[R]}_L(\pi)$, which, in turn is the undirected random graph $G(r,\frac{1}{2})$. The expected number of edges of ${Q[R]}_L(\pi)$ is therefore $\frac{1}{2}\binom{r}{2}$ and the probability of this number being smaller than
	$\frac{1}{4}\binom{r}{2}$ (which is precisely half the expectation)  is, by Chernoff's inequality, at most $\exp(-\frac{1}{8}\binom{r}{2})$.
	As there are only $r!$ possibly $\pi$ to consider, we have that the probability that
	${\rm inv}_2(Q[R]) < \frac{1}{4}\binom{r}{2}$ is at most $r!\exp(-\frac{1}{8}\binom{r}{2})$.
	Now, for any $r$, the number of possible subsets $R$ is $\binom{q}{r}$, so the probability that
	(ii) fails is at most
	$$
		\Pr[(ii) ~{\rm fails}] \le \sum_{r=\lceil\log_{1.04} q\rceil}^{q} \binom{q}{r} r!e^{-\frac{1}{8}\binom{r}{2}}
		\le q (qr)^r e^{-\frac{1}{8}\binom{r}{2}} \le 1.05^{r^2} e^{-\frac{1}{8}\binom{r}{2}}  \ll \frac{1}{3}
	$$
where we have used that $1.04^r \ge q$, that $1.05e^{-1/16} < 1$ and that $k$ (hence $q$, hence $r$) is sufficiently large.
\end{proof}

Let $R$ be an $r$-vertex tournament and suppose that $n$ is a multiple of $r$.
An $n$-vertex {\em balanced blowup} of $R$ is obtained by replacing each vertex $i \in V(R)$
with a set $V_i$ of size $n/r$ (all the $V_i$'s are pairwise disjoint), and constructing a tournament
with vertex set $\cup_{i \in V(R)} V_i$ as follows. For each edge $(i,j) \in E(R)$, all edges are oriented from $V_i$ to $V_j$ (we call such edges {\em outer}), and for each $V_i$, all edges with both endpoints in $V_i$ are oriented arbitrarily (we call such edges {\em inner}).
\begin{lemma}\label{l:blowup}
	Suppose $R$ is an $r$-vertex tournament having ${\rm inv}_2(R)=m$ and let $Z$ be an $n$-vertex balanced blowup of a $R$.
	Then for any permutation $\sigma$ of $V(Z)$, the graph $Z_L(\sigma)$ contains at least
	$mn^2/r^2$ outer edges.
\end{lemma}
\begin{proof}
	Consider some permutation $\sigma$ of $V(Z)$ and the corresponding $Z_L(\sigma)$. Recall that $V(Z)$ is the disjoint union of $r$ sets $V_i$ for $i \in V(R)$ and let $W$ be a transversal of the $V_i$'s.
	Then, by construction, $Z[W]$ is isomorphic to $R$. So, the subgraph of $Z_L(\sigma)$ induced by $W$
	contains at least ${\rm inv}_2(R)=m$ outer edges. By double-counting over all possible transversals, we have
	that $Z_L(\sigma)$ contains at least $mn^2/r^2$ outer edges.
\end{proof}

Fix a tournament $Q$ on vertex set $[q]$ where $q=k^2$, satisfying both items of Lemma \ref{l:q}.
Let $T$ be an $n$-vertex balanced blowup of $Q$ (we may assume than $n$ is a multiple of $q$ as removing a constant number of vertices from a tournament $T$ on $n$ vertices only decreases ${\rm inv}_k(T)$ by $o(n^2)$ so does not affect the asymptotic claim of Theorem \ref{t:inv-k}).
We will show that ${\rm inv}_k(T)$ is at least as large as the claimed lower bound in Theorem \ref{t:inv-k}.

\begin{proof}[Proof of Theorem \ref{t:inv-k}, lower-bound]
Suppose $\{X_1,\ldots,X_t\}$ is a $k$-decycling set of $T$ where $t={\rm inv}_k(T)$.
Listing the elements of $T$ in the order they appear in the resulting transitive tournament corresponds to some permutation $\pi$ of $V(T)$ for which all edges of $T_L(\pi)$ have been reversed.
We will show that we need $t$ to be rather large in order to reverse all edges of $T_L(\pi)$ for {\em any} possible $\pi$. So, fix some such $\pi$ and partition the edges of $T_L(\pi)$ into two parts $E_{\rm in}$ and
$E_{\rm out}$ where the former are the edges of $T_L(\pi)$ that are inner edges of $T$ and the latter are the edges of $T_L(\pi)$ that are outer edges of $T$.

We next lower-bound $|E_{\rm out}|$. By Lemma \ref{l:blowup}, using $R=Q$, $r=q$, $m \ge q^2/4-2q^{3/2}$ (which follows from Lemma \ref{l:q} Part (i)) $Z=T$ and $\sigma=\pi$, we have
\begin{equation}\label{e:out}
|E_{\rm out}| \ge \frac{n^2}{q^2}\left(\frac{q^2}{4}-2q^{3/2} \right) = \frac{n^2}{4}-\frac{2n^2}{k}.
\end{equation}

Now, consider some $X$ from the aforementioned $k$-decycling set.
We will show that $X$ reverses at most $\binom{k}{2}- \Omega(k^2/\log^4 k)$ edges of $E_{\rm out}$.
Let $|X|=x$ and suppose first that $x \le k/2$. In this case we can use the trivial fact that $X$ reverses at most $\binom{x}{2} \le k^2/4$ edges. Assume therefore that $k/2 \le x \le k$.
Recall that $V(T)$ consists of $q$ parts $V_1,\ldots,V_q$.
Let $W_i = V_i \cap X$ for $1 \le i \le q$.
Suppose next that some $W_i$ has size at least $k/\log^2 k$. In this case, we see that $X$ induces at least $\binom{|W_i|}{2}$ inner edges, implying that $X$ reverses at most $\binom{k}{2}- \Omega(k^2/\log^4 k)$ edges of $E_{\rm out}$.
Thus, we may assume that $0 \le |W_i| \le k/\log^2 k$ for all $1 \le i \le q$.

Partition the $W_i$ into {\em bunches} according to their size.
For $1 \le j \le \lfloor \log k \rfloor$ we say that $W_i$ is in bunch $j$ if $2^{j-1} \le |W_i| < 2^j$.
Letting $B_j$ be the union of all the $W_i$ in bunch $j$, we have
$\sum_{j=1}^{\lfloor \log k \rfloor} |B_j| = x$.
Let $j^* = \lfloor \log k - 3\log \log k \rfloor$. Consider first the case that $\sum_{j=1}^{j^*} |B_j| \le x/2$.
In this case, the last $\lceil 3\log \log k \rceil$ bunches contain together at least $x/2$ vertices of $X$.
But since in each such bunch it holds that $|W_i| \ge 2^{j^*-1} = \Omega(k/\log^3 k)$,
it induces at least $\Omega(k^2/\log^6 k)$ inner edges. Furthermore, as each $|W_i| \le k/\log^2 k$
and since $x/2 \ge k/4$,  there are $\Omega(\log^2 k)$ such $W_i$ in the last $\lceil 3\log \log k \rceil$ bunches, so together they induce
at least $\Omega(\log^2 k \cdot k^2/\log^6 k) = \Omega(k^2/\log^4 k)$ inner edges,
implying that $X$ reverses at most $\binom{k}{2}- \Omega(k^2/\log^4 k)$ edges of $E_{\rm out}$.

We remain with the case that $\sum_{j=1}^{j^*} |B_j| \ge x/2$, 
so there is some bunch $j$ with $1 \le j \le j^*$ for which $|B_j| \ge x/2j^* \ge x/2 \log k$,
and we shall focus on that bunch. Notice that since each $W_i$ in bunch $j$ has size at most
$2^j$, this means that the number of $W_i$ in bunch $j$ is at least
$$
\frac{x}{2^{j+1} \log k} \ge \frac{k}{2^{j+2} \log k} \ge \frac{k}{2^{j^*+2} \log k} \ge \frac{\log^2 k}{4}\;.
$$
Let $R = \{i \,:\, W_i {\rm~is~in~bunch~} j \}$, noticing that $R \subseteq [q]=V(Q)$ and that
$r=|R| \ge \frac{k}{2^{j+2} \log k}  \ge \frac{\log^2 k}{4}$ by the last inequality.
Observe that if $i \in R$ then $|W_i| \ge 2^{j-1}$, so let $W_i^* \subseteq W_i$ be chosen such that
$|W_i^*|=2^{j-1}$. By Lemma \ref{l:q}, item (ii), we have that ${\rm inv}_2(Q[R]) \ge \frac{1}{4}\binom{r}{2}$.
Notice also that the union of the $W_i^*$ for $i \in R$ is an $r2^{j-1}$ balanced blowup of $Q[R]$.
Hence, by Lemma \ref{l:blowup}, with $R=R[Q]$, $m \ge \frac{1}{4}\binom{r}{2}$, $Z$ being the subgraph of
$T[X]$ induced by the union of the $W_i^*$, and $\sigma=\pi^{reverse}$ we have that $Z_L(\sigma)$
contains at least $\frac{1}{4}\binom{r}{2}2^{2j-2}$ outer edges. Notice that all of these outer edges
do not appear in $T_L(\pi)$ so are not reversed by $X$. It follows that $X$ reverses at most
$$
\binom{k}{2} - \frac{1}{4}\binom{r}{2}2^{2j-2} \le \binom{k}{2} - \frac{1}{10}\left(\frac{k}{2^{j+2}\log k}\right)^2 2^{2j-2} = \binom{k}{2} - \Omega(k^2/\log^2 k)
$$
edges of $E_{\rm out}$.

It now follows from \eqref{e:out} and form the fact that each $X$ element in the decycling set reverses at most
$\binom{k}{2}- \Omega(k^2/\log^4 k)$ edges of $E_{\rm out}$, that the number of elements of the decycling set must
be
$$
{\rm inv_k}(n) \ge {\rm inv_k}(T) \ge \frac{\frac{n^2}{4}-\frac{2n^2}{k}}{\binom{k}{2}- \Omega(k^2/\log^4 k)}\;.
$$
The existence of $\delta_k > 0$ for all $k \ge k_0$ is now guaranteed since for all sufficiently large $k$,
$$
\frac{\frac{1}{4}-\frac{2}{k}}{\binom{k}{2}- \Omega(k^2/\log^4 k)} > \frac{1}{2k(k-1)}\;.
$$
\end{proof}

We now turn to proving the upper bound in Theorem \ref{t:inv-k}.
We start with the following lemma which is analogous to Lemma \ref{l:c4}.
\begin{lemma}\label{l:kk}
	Let $T^*$ be a tournament and suppose that $k \ge 4$ is even and $T^*_L(\pi)$ has a (not necessarily induced) copy of $K_{k,k}$ or that $k \ge 5$ is odd and $T^*_L(\pi)$ has a (not necessarily induced) copy of $K_{k+1,k-1}$. Let $H$ denote the corresponding copy. Then, there are four sets of vertices $X,Y,Z,W$ of $k$ vertices each, such that inverting $\{X,Y,Z,W\}$ reverses the direction of the edges of $H$ without
	affecting the direction of any other edge of $T^*$.
\end{lemma}
\begin{proof}
	Suppose first that $k$ is even and that $H$ is a copy of $K_{k,k}$. Consider the bipartition of $H$, denoting the vertices of one part by $v_1,\ldots,v_k$ and the other part by $u_1,\ldots,u_k$.
	Let $A=\{v_1,\ldots,v_{k/2}\}$, $B=\{v_{k/2+1},\ldots,v_k\}$, $C=\{u_1,\ldots,u_{k/2}\}$, $D=\{u_{k/2+1},\ldots,u_k\}$. Inverting $\{A \cup C, A \cup D, B \cup C, B \cup D\}$ reverses the direction of the $k^2$ edges of $H$
	without affecting the direction of any other edge of $T^*$.
	Suppose next that $k$ is odd and that $H$ is a copy of $K_{k+1,k-1}$. Consider the bipartition of $H$, denoting the vertices of one part by $v_1,\ldots,v_{k+1}$ and the other part by $u_1,\ldots,u_{k-1}$.
	Let $A=\{v_1,\ldots,v_{(k+1)/2}\}$, $B=\{v_{(k+1)/2+1},\ldots,v_{k+1}\}$, $C=\{u_1,\ldots,u_{(k-1)/2}\}$, $D=\{u_{(k-1)/2+1},\ldots,u_k\}$. Inverting $\{A \cup C, A \cup D, B \cup C, B \cup D\}$ reverses the direction of the $k^2-1$ edges of $H$
	without affecting the direction of any other edge of $T^*$.
\end{proof}
The next lemma is a simple consequence of Tur\'an's Theorem and Ramsey's Theorem.
\begin{lemma}\label{l:tr}
	For every fixed $k \ge 3$ there exists $\gamma_k > 0$ such that in any $2$-coloring of the edges of
	$K_n$, there are at least $(1-o(1))\gamma_k n^2$ pairwise edge-disjoint monochromatic copies of $K_k$, all of the same color.
\end{lemma}
\begin{proof}
	Let $q=R(k) < 4^k$ be the diagonal Ramsey number of $k$. Consider some $2$-coloring of the edges of
	$K_n$, and remove monochromatic edge-disjoint copies of $K_k$ until none are left. We must have removed at least $n^2/q^2$ edges as otherwise by Tur\'an's Theorem, there is a $K_q$ on the non-removed edges, so a monochromatic $K_k$. The result follows for a suitable choice of $\gamma_k$ as at least half of the monochromatic removed $K_k$ are of the same color.
\end{proof}

\begin{proof}[Proof of Theorem \ref{t:inv-k}, upper-bound]
Let $T$ be an $n$-vertex tournament. Considering some random permutation $\pi$ of $V(T)$, we have that with high probability, $|E(T_L(\pi))| = (1+o(1))n^2/4$ and $|E(T_R(\pi))| = (1+o(1))n^2/4$.
By Lemma \ref{l:tr} one of $T_L(\pi)$ or $T_R(\pi)$ has at least  $(1-o(1))\gamma_k n^2$ pairwise edge-disjoint copies of $K_k$. Without loss of generality, assume this is $T_L(\pi)$. Removing a set of
$(1-o(1))\gamma_k n^2$ edge-disjoint copies of $K_k$ from $T_L(\pi)$ amounts to inverting this amount of $k$-sets of vertices,
such that after applying these inversions we obtain a tournament $T^*$ such that $T^*_L(\pi)$ has at
most $(1+o(1))n^2/4 - (1-o(1))\binom{k}{2}\gamma_k n^2$ edges.
By Lemma \ref{l:kk} we can repeatedly invert quartets of $k$-sets of vertices each, 
until we obtain a tournament $T^{**}$ for which $T^{**}_L(\pi)$ has no $K_{k,k}$ (when $k$ is even) or no $K_{k+1,k-1}$ (when $k$ is odd). The number of inversions performed starting at $T^*$ and arriving at $T^{**}$ is
therefore precisely $4(|E(T^*_L(\pi))|-|E(T^{**}_L(\pi))|)/k^2$ if $k$ is even and precisely
$4(|E(T^*_L(\pi))|-|E(T^{**}_L(\pi))|)/(k^2-1)$ if $k$ is odd.
By the Kov\'ari-S\'os-Tur\'an Theorem \cite{KST-1954}, we have $|E(T^{**}_L(\pi))| =O(n^{2-1/k})$.
We therefore obtain that
$$
{\rm inv}_k(T) \le (1+o(1))n^2 \left [\gamma_k + \frac{4}{k^2}\left(\frac{1}{4}-\binom{k}{2}\gamma_k\right)\right] =
(1+o(1))n^2 \left[ \frac{1}{k^2}-\epsilon_k \right]
$$
when $k$ is even and
$$
{\rm inv}_k(T) \le (1+o(1))n^2 \left [\gamma_k + \frac{4}{k^2-1}\left(\frac{1}{4}-\binom{k}{2}\gamma_k\right)\right] =
(1+o(1))n^2 \left[ \frac{1}{k^2-1}-\epsilon_k \right]
$$
for a suitable $\epsilon_k > 0$.
As $T$ is an arbitrary $n$-vertex tournament, we obtain (unifying the even and odd cases of $k$) that
$$
{\rm inv}_k(n) \le 
(1+o(1))n^2 \left[ \frac{1}{2\lfloor k^2/2\rfloor}-\epsilon_k \right]\;.
$$

\end{proof}

\end{document}